\documentclass[11pt,leqno]{amsart}
\usepackage{amsthm,amsfonts,amssymb,amsmath,oldgerm}
\numberwithin{equation}{section}
\usepackage{fullpage}
\usepackage{color}
\usepackage{calrsfs}
\DeclareMathAlphabet{\pazocal}{OMS}{zplm}{m}{n}

\newcommand{\ks}{k_*}

\def\eps{\varepsilon }

\newcommand\R{\mathbb R}

\def\eps{\varepsilon}

\newcommand\br{\begin{remark}}
\newcommand\er{\end{remark}}
\newcommand\bp{\begin{pmatrix}}
\newcommand\ep{\end{pmatrix}}
\newcommand{\be}{\begin{equation}}
\newcommand{\ee}{\end{equation}}
\newcommand\ba{\begin{equation}\begin{aligned}}
\newcommand\ea{\end{aligned}\end{equation}}

\newcommand{\bap}{\begin{app}}
\newcommand{\eap}{\end{app}}
\newcommand{\begs}{\begin{exams}}
\newcommand{\eegs}{\end{exams}}
\newcommand{\beg}{\begin{example}}
\newcommand{\eeg}{\end{exaplem}}
\newcommand{\bpr}{\begin{proposition}}
\newcommand{\epr}{\end{proposition}}
\newcommand{\bt}{\begin{theorem}}
\newcommand{\et}{\end{theorem}}
\newcommand{\bc}{\begin{corollary}}
\newcommand{\ec}{\end{corollary}}
\newcommand{\bl}{\begin{lemma}}
\newcommand{\el}{\end{lemma}}
\newcommand{\bd}{\begin{definition}}
\newcommand{\ed}{\end{definition}}
\newcommand{\brs}{\begin{remarks}}
\newcommand{\ers}{\end{remarks}}

\newcommand{\CC}{{\mathbb C}}

\newcommand{\Id}{{\rm Id }}

\newtheorem{theorem}{Theorem}[section]
\newtheorem{proposition}[theorem]{Proposition}
\newtheorem{corollary}[theorem]{Corollary}
\newtheorem{lemma}[theorem]{Lemma}

\theoremstyle{remark}
\newtheorem{remark}[theorem]{Remark}
\theoremstyle{definition}
\newtheorem{definition}[theorem]{Definition}

\newtheorem{example}[theorem]{Example}

\newcommand\cA{{\mathcal A}}

\newcommand\cJ{{\mathcal J}}

\newcommand\cR{{\mathcal R}}

\newcommand\cL{{\mathcal L}}
\newcommand\cN{{\mathcal N}}

\newcommand{\RM}{\mathbb{R}}

\newcommand{\beq}{\begin{equation}}
\newcommand{\eeq}{\end{equation}}

\title{Forward-modulated damping estimates and nonlocalized stability of periodic Lugiato-Lefever waves}
\author{Kevin Zumbrun}
\address{Indiana University, Bloomington, IN 47405}
\email{kzumbrun@indiana.edu}
\thanks{Research of K.Z. was partially supported under NSF grant no. DMS-0300487}

\begin{document}

\begin{abstract}
In an interesting recent analysis, Haragus-Johnson-Perkins-de Rijk have shown
modulational stability under localized perturbations of steady periodic solutions of the 
Lugiato-Lefever equation (LLE), in the process pointing out a difficulty in obtaining standard ``
nonlinear damping estimates'' on modulated perturbation variables to control regularity of solutions.
Here, we point out that in place of standard ``inverse-modulated'' damping estimates, one can alternatively 
carry out a damping estimate on the ``forward-modulated'' perturbation, noting that norms of forward- and 
inverse-modulated variables are equivalent modulo absorbable errors, thus recovering the classical argument 
structure of Johnson-Noble-Rodrigues-Zumbrun for parabolic systems.  
This observation seems of general use in situations of delicate regularity.
	Applied in the context of (LLE) it gives the stronger result of stability and asymptotic
	behavior with respect to nonlocalized perturbations.
\end{abstract}

\maketitle

\section{Introduction}\label{s:intro}

In the interesting recent work \cite{HJPR}, building on linear analysis in \cite{HJP},
Haragus, Johnson, Perkins, and de Rijk study nonlinear modulational stability of steady periodic solutions of
the Lugiato-Lefever equations (LLE), a model for pattern formation in an optical medium in a cavity 
under excitement by laser pumping.
A general framework for the passage from linear to nonlinear modulational stability has been
set up in \cite{JZN,JNRZ3} and related works, and the authors loosely follow this path.
However, they find it necessary to modify the approach substantially in the treatment of regularity,
substituting for the usual nonlinear modulational damping a combination of ``tame'' unmodulated estimates
with exponentially decaying linear terms, and using in an important way semilinearity of the
underlying equations (LLE).
This strategy, introduced in \cite{RS} in the context of multi-D stability of planar periodic wave trains,
is shown in \cite{HJPR} to be sufficient also to resolve the regularity issues arising in treatment of 1-D stability
of Lugiato-Lefever waves,
thus adding a new and useful element in the toolkit for modulational stability, applicable to the semilinear case.

The authors in passing pose the question whether a modulated damping estimate is obtainable at all for
(LLE).
And, indeed, this is not just of academic interest.
For, the unmodulated estimates obtained in \cite{HJPR}, though sufficient to close their stability
argument, are substantially weaker than those that would be given by a standard modulated damping estimate.
Specifically, as we discuss in detail below, the argument of \cite{HJPR} gives control of the $H^{s_0}$ norm
of the modulated perturbation $v$ 
by its $L^2$ norm times an arbitrarily slowly growing algebraic factor,
 modulo the $L^2$ norm of the modulating phase perturbation $\gamma$.
By contrast, a standard damping estimate controls the $H^{s_0}$ norm of $v$ by its (exact) $L^2$ norm,
modulo the $L^2$ norm of {\it the derivative of $\gamma$.}

A central point in the study of modulational stability is that $v$ and derivatives of the phase perturbation 
$\gamma$ exhibit comparable decay, whereas $\gamma$ itself decays more slowly, by a factor of $(1+t)^{1/2}$.
Indeed, this is the motivation for separating out phase perturbation in the first place, 
with the goal being to obtain a nonlinear iteration involving only faster-decaying
$v$ and $\partial_{x,t}\gamma$, hence more likely to close \cite{JZ}.
Thus, {\it the $L^2 \to H^{s_0}$ control afforded by damping modulo $\|\partial_x \gamma\|_{L^2}$
is in principle sharper by factor $(1+t)^{-1/2}$ than that afforded modulo $\|\gamma\|_{L^2}$
by the strategy of \cite{HJPR}.}

Heuristically, the phase perturbation is expected to satisfy an approximate Burgers equation \cite{HK,DSSS}
hence to decay approximately as solutions of the heat equation,
a property that has been validated rigorously in various settings in \cite{JNRZ1,JNRZ2,SSSU},
In the case considered in \cite{HJPR} 
of {\it localized initial perturbations} $\tilde v$ of a background periodic wave $\bar u$, one has, writing
$u=\bar u+ \tilde v$ and defining the modulated perturbation variable 
$$
v(x,t)= u(x+\gamma(x,t),t)-\bar u(x)
$$
that $v\sim \tilde v + \bar u_x \gamma$, with $\bar u_x \sim 1$, hence $L^1$ localization
$\|v\|_{L^1}, \|\tilde v\|_{L^1}<\infty$ on $v, \tilde v$ imposes $L^1$ localization 
$\|\gamma\|_{L^1}<+\infty$ on $\gamma$ as well, leading to decay rate 
$\|\gamma(\cdot, t)\|_{L^2} \lesssim (1+t)^{-1/4}$.
Thus, one expects $\|\tilde v\|_{L^2} \sim \|\gamma\|_{L^2} \lesssim (1+t)^{-1/4}$ and
$\|v\|_{L^2} \sim \|\partial_{x,t}\gamma\|_{L^2}  \lesssim (1+t)^{-3/4}$.
And, indeed, this is the result proved in \cite{HJPR} for initial perturbations $\tilde v_0$
sufficiently small in $L^1\cap H^4$ of a smooth spatially periodic standing-wave solution of (LLE)
satisfying the standard {\it diffusive spectral stability} condition of Schneider \cite{S1,S2,JZ,JNRZ1,JNRZ2}.
Though not stated in the theorem, the main derivative bounds obtained in the proof are
$\|v\|_{H^2}\lesssim 1$ and $\|v\|_{H^4}\lesssim (1+t)^{1/4}$,
with $H^{s_0}$ corresponding to $H^2$ being the crucial bound needed to close the nonlinear iteration.
Thus, the respective error terms $\|\gamma\|_{L^2}$ and
$\|\partial_x \gamma\|_{L^2}$, being $\ll 1$ are both irrelevant, and so the
difference in the context of this argument between tame and damping type estimates is mainly in
simplification/standardization of the argument and not in the finally obtained result.\footnote{
	The simplification afforded by damping, however, is rather great, eliminating the need
	for integration by parts and mean value inequalities and most of the technical complications of the proof
	(cf. \cite{JZ,HJPR}).}

For {\it nonlocalized initial perturbations} on the other hand, as considered in \cite{JNRZ1,JNRZ2,SSSU,JNRZ3},
the phase perturbation $\gamma$ is taken merely bounded in $L^\infty$, with 
$L^1$ localization $\|\partial_x \gamma\|_{L^2}<\infty$ imposed, rather, on its derivative.
This yields decay rates $\|v\|_{L^2}, \|\partial_x \gamma\|_{L^2} \lesssim (1+t)^{-1/4}$,
but with $\|\gamma\|_{L^2}=+\infty$.
Thus, in this context, $\|\gamma\|_{L^2}$ is clearly {\it not} a negligible error, and so the argument
structure of \cite{HJPR} based on tame estimates of the unmodulated perturbation $\tilde v$ does not
suffice to close a nonlinear iteration.
The sharper bound $\|v\|_{H^{s_0}}\sim \|\partial_x \gamma\|_{L^2} \lesssim  (1+t)^{-1/4}$ afforded by
damping estimates, however, does suffice, yielding both stability and asymptotic behavior.

\medskip

{\it This gives substantial motivation, of practical interest, to answering the question of \cite{HJPR}
whether or not it is possible to obtain a modulated damping estimate  for (LLE).}

\medskip

We are not able to answer this question in the sense of the original damping estimates formulated in
\cite{JNRZ3} and elsewhere.
However, in the present brief note, motivated by the discussion of \cite{HJPR},
we point out that (i) a modulated damping estimate {\it can} be obtained for (LLE) if one substitutes for
the usual ``inverse-modulated'' perturbation variable \cite{JNRZ3} a ``forward-modulated'' version
(see below for definitions), and (ii) inverse- and forward-modulated perturbation variables are equivalent
in all relevant norms, modulo an absorbable high-derivative $H^r$ norm of $\gamma_x$.

This gives an alternative approach to the regularity problem for (LLE) in the original spirit of \cite{JNRZ3} 
(see final section), which (i) removes much of the technical complexity of the argument, and
(ii) is not inherently limited to the semilinear case.
More important, our approach yields significantly sharper bounds, allowing the treatment as in \cite{JNRZ3}
of {\it nonlocalized perturbations} allowing different asymptotic limits of the phase shift $\gamma$, 
whereas the argument of \cite{HJPR} requires $\|\gamma\|_{L^2}<\infty$.
Thus, we obtain at once a substantial simplification of the argument and a substantial generalization of the result.

We emphasize that our approach is not tied to (LLE) but applies for arbitrary systems.
Indeed, we would propose as a useful option, substituting for the original framework of \cite{JNRZ3}
the modified one of obtaining linear bounds in inverse-modulated variables, but damping-type energy estimates
in forward-modulated ones where they may be easier to obtain.
We believe this observation to be of general use in situations of delicate regularity,
belonging in the multi-purpose toolkit of \cite{JNRZ3}.

\medskip
{\bf Acknowledgement:} We thank the referees for their careful reading and helpful suggestions.

\section{Comparisons of techniques}\label{s:comparison}

We begin by comparing the various techniques in a general context.

\subsection{Damping vs. tame estimates}\label{s:comp}
A standard issue in the approach of \cite{JNRZ3} is closing a nonlinear
iteration despite apparent loss of derivatives in the nonlinear (modulational) perturbation equations
(displayed below for (LLE)).  This has previously been addressed by the use of 
{\it nonlinear damping estimates} \cite[Section 1.3]{JNRZ3}
\be\label{damp}
\partial_t \mathcal{E}(v)\leq -\eta \mathcal{E}(v) + C(\|v\|_{L^2_\alpha}^2 + \|\partial_{x,t}\gamma\|_{H^r_\alpha}),
\ee
$s<r$, where $\mathcal{E}(v)$ is an energy controlling a (possibly weighted, with weight denoted by $\alpha$) 
Sobolev norm $\|v\|_{H^s_\alpha}^2$ for the modulated perturbation 
$$
v(x,t)=u(x+\gamma(x,t), t)-\bar u(x),
$$
where $u$ and $\bar u$ are perturbed and background solutions, and $\gamma$ (see below) is a modulation
parameter introduced in the analysis with arbitrarily high derivative control of the same order as
$\|v\|_{L^2_\alpha}$.
This yields by a Gronwall-type estimate control of $\|v\|_{H^s_\alpha}^2$ by 
$e^{-\eta t} \|v\|_{H^s_\alpha}^2(0)$ plus the integral of an exponentially decaying memory kernel against
$C(\|v\|_{L^2_\alpha}^2 + \|\partial_{x,t}\gamma\|_{H^r_\alpha})$,
thus effectively controlling $H^s$ by $L^2$ decay.

In \cite{HJPR}, the authors use an alternative approach introduced in \cite{RS},
playing modulated and unmodulated perturbations against each other to obtain a result.
The ingredients needed are exponential decay of high-frequency linear estimates for the unmodulated semigroup, plus
the aforementioned semilinear structure, allowing the unmodulated perturbation to be estimated 
via Duhamel's principle thanks to the fact that there is no loss of derivatives in the unmodulated equation..
Specifically, one attains on the unmodulated variable 
$$
\tilde v(x,t)=u(x,t)-\bar u(x)
$$
the tame estimate
$ \|\tilde v\|_{H^s}(t)\leq C(1+t)^{1/4}, $
for arbitrary $s$ so long as (i) $\|\tilde v\|_{H^{s_0}}$ remains small for some fixed $s_0$ ($s_0=2$ in 
the argument of \cite{HJPR}) and (ii) the undifferentiated unmodulated variable 
decays at the rate $\|\tilde v\|_{L^2}(t)\leq C(1+t)^{-1/4}$ predicted by linear theory.
By Sobolev interpolation, taking $s$ high enough, one may estimate 
\be\label{dopt}
\|\tilde v\|_{H^{s_0}}(t) \leq C(1+t)^{-1/4+\eps}
\ee
for $\eps>0$ as small as desired, nearly recovering the decay rate of $\|v\|_{L^2}(t)$.

This argument is closed by ``mean value inequalities'' \cite[Lemma 4.9]{HJPR}
controlling $L^2$ norms of 
$ v-\tilde v= u(x+\gamma,t)-u(x,t) \sim u_x \gamma $
and $ \partial_x^{s_0}(v-\tilde v)= \sim \partial_x^{s_0+1}u \gamma $
by constant multiples of $\|\gamma\|_{L^2}$, together with integration by parts formulae
\cite[Lemma 4.8]{HJPR} effectively shifting derivatives in the Duhamel formulation for $v$ 
from $v$-factors to harmless $\gamma$-factors in order to minimize the required bounds on $v_x$.


\subsection{Forward-modulated damping} Here, we observe that the forward-modulated perturbation equation
(\eqref{fmod} below) like the unmodulated one, involves no loss in derivatives, hence admits a damping estimate
modulo higher-derivative terms in the modulation parameter $\gamma$.
The forward- and inverse-modulated perturbation variables can be seen to decay at the same rates (Section 
\ref{s:vs} below), for general choice of system. 
Thus, the substitution of forward-modulated damping estimates for the unmodulated estimate of \cite{HJPR}
would appear both to streamlines the argument a bit, and to apply in principle to a wider class of systems.
(Indeed, as noted below, it can be applied in all cases of delicate regularity \cite{HJPR,Z,JZN} treated so far.)

\medskip

In the remainder of the paper, we give technical details filling in the outline above for (LLE).

\section{Nonlinear perturbation system}\label{s:npert}

The Lugiato-Lefever equation (LLE) is
\be\label{1.1}
\partial_t \psi= -i\beta \partial_x^2 \psi - (1+i\alpha)\psi + i|\psi|^2\psi + F
\ee
with $\psi\in \CC$ and parameters $\alpha, \beta, F\in \R$ and $F>0$.
Perturbing about a steady periodic solution $\phi$, with $\psi=\phi+\tilde v$,
and expanding $\tilde v=\tilde v_r+ i\tilde v_i$ in real and imaginary parts gives, following \cite{HJPR},
\be\label{1.3}
\partial_t \bp \tilde v_r\\\tilde v_i\ep=
\cA[\phi]\partial_t \bp \tilde v_r\\\tilde v_i\ep + \cN(\tilde v),
\ee
where $\cN$ is quadratic order in $\tilde v$ and
\be\label{1.4}
\cA[\phi]=-\Id + \cJ \cL[\phi],
\ee
with
\be\label{JL}
\cJ=\bp 0 & -1\\1 & 0\ep, \qquad
\cL[\phi]=\bp -\beta \partial_x^2 -\alpha + 3\phi_r^2 + \phi_i^2 & 2\phi_r\phi_i\\
2\phi_r\phi_i &    -\beta \partial_x^2 -\alpha + \phi_r^2 + 3\phi_i^2 
\ep.
\ee

\section{Unmodulated damping estimate}\label{s:udamp}
Following \cite{HJPR}, define the (unmodulated) energy
$$
\tilde E_j (t) =
\|\partial_x^j \tilde v  \|_{L^2}^2 - \frac{1}{2\beta}
\langle \cJ M[\phi]\partial_x^{j-1}\tilde v, \partial_x^{j-1}\tilde v \rangle,
$$
\be\label{M}
M[\phi]:= 2 \bp -2\phi_r\phi_i &  \phi_r^2-\phi_i^2\\ \phi_r^2-\phi_i^2 & 2\phi_r \phi_i\\ \ep
\ee
yielding, after some computation \cite[Appendix A]{HJPR}
$$
\partial_t \tilde E_j (t) = -2\tilde E_j (t) + R_1(t) + R_2(t),
$$
where $R_1(t)$, comprising lower-order derivative bilinear terms, satisfies
$$
|R_1(t) |\leq C_1 \Big( \|\partial_x^{j-1}\tilde v\|_{L^2} + \|\tilde v\|_{L^2}\Big)(t),
$$
and $R_2$ is a nonlinear residual, satisfying
$$
|R_2(t)|\leq C_2 \|\partial_x^j\tilde v\|^2_{L^2}(t)
\Big( \|\partial_x^{j}\tilde v\|_{L^2} + \|\tilde v\|_{L^2} \Big)(t).
$$

Combining, and using Sobelev interpolation, one obtains for $\|\partial_x^j\tilde v\|_{L^2}$ sufficiently small
the unmodulated nonlinear damping estimate
\be\label{nldamp}
\partial_t \tilde E(t)\leq - \theta \tilde E(t) +C\|\tilde v\|_{L^2}^2(t), \qquad \theta>0,
\ee
yielding after integration
\be\label{NLgron}
\tilde E(t)\leq e^{-\theta t}\tilde E(0)+ C\int_0^t e^{-\theta(t-s)}\|\tilde v\|_{L^2}^2(s) ds,
\ee
so long as $\|\partial_x^j\tilde v\|_{L^2}(s)$ is sufficiently small on $0\leq s\leq t$, hence,
applying Sobolev interpolation once more, controlling
$\|\partial_x^j\tilde v(t)\|_{L^2}$ by  $\frac12\tilde E(t)+ C\|\tilde v\|_{L^2}(t)$, 
and thus ultimately by $\|\tilde v\|_{L^2}(t)$.

\br\label{drmk}
As noted in \cite[Rmk. 1.4]{HJPR}, this unmodulated damping estimate can substitute for the tame estimates
of \cite{HJPR} with little change in the argument, the advantage being the possibility 
(to be checked in individual cases) of extension to the quasilinear case.
\er

\section{Modulated damping estimates}\label{s:modpert}

\subsection{Inverse modulation}\label{s:inversemod}
We first recall the standard ``centered'' or ``inverse-modulated'' perturbation
\be\label{stdmod}
v(x,t):=\psi(x+ \gamma(x,t),t) -\phi(x)\approx \tilde v + (\phi_x + \tilde v_x) \gamma
\ee
serving as the primary perturbation variable in \cite{HJPR}, where the phase-modulation $\gamma$
is chosen in nonlocal fashion so as to to remove the principal time-asymptotic part of $\tilde v$,
thus minimizing $v$.

As the choice of $\gamma$ concerns long-time behavior, there is a great deal of flexibility in its
short-time behavior- in particular, {\it $\gamma$ may be chosen so that it and all derivatives are
bounded in short time and decaying at optimal linear rate in long time.}
See, e.g., \cite[\S 2.2]{JNRZ3}, for a general description of this strategy.
Writing
$$
\psi(x,t)=(\phi + v)(x-\gamma,t),
$$
computing derivatives
$$
\begin{aligned}
	\partial_t \psi(x,t)&=(\phi_t+v_t)(x-\gamma,t) - (\phi_x+v_x)(x-\gamma,t) \partial_t \gamma,\\
	\partial_x \psi(x,t)&=(\phi_x+v_x)(x-\gamma,t) - (\phi_x+v_x)(x-\gamma,t) \partial_x \gamma,\\
\end{aligned}
$$
etc., and substituting into \eqref{1.1}, yields, after a computation as in \eqref{1.3},
a $v$-equation consisting of the one for $\tilde v$ together with new terms 
$(\phi_x+v_x))\partial_x\gamma$ and $(\phi_x+v_x))\partial_t\gamma$ and their derivatives in $x$.

Terms involving only $\phi$ and $\gamma$ are harmless, as derivatives of $\phi$ are bounded and derivatives
of $\gamma$ are of the same order in $L^2$ as $v$ itself \cite{JZN,HJPR}.
However, a persistent issue in problems without parabolic smoothing is the appearance of 
terms involving products of highest-derivative $\phi$ and $v$ terms.
These can sometimes be treated by judicious rearrangement/construction of energy functional \cite{Z,JZN};
however, even when it succeeds, this can cost a great deal of additional effort.

In the present case (LLE), as pointed out in \cite{HJPR}, inverse modulation changes the perturbation equations
from semilinear to quasilinear form, seemingly preventing such a nonlinear damping estimate altogether;
see discussion, \cite[App. A]{HJPR}.
For this reason, the authors restrict to 
tame estimates on the slower-decaying but favorable regularity unmodulated
variable $\tilde v$, coupling this via an auxiliary argument as in \cite{RS} to their linearized bounds 
on the inverse-modulated perturbation $v$ to obtain ultimately, optimal nonlinear bounds on $\|v\|_{L^p}$.
As they point out, they could alternatively substitute a damping estimate on the unmodulated variable $\tilde v$.

\subsection{Forward modulation}\label{s:rmodpert}
A more natural modulated variable in many ways is the forward-modulated variable
\be\label{fmod}
\bar v:= \psi(x,t)- \phi(x-\gamma, t)\approx 
\tilde v + \phi_x(x-\gamma,t)  \gamma \approx v - \phi_{xx}\gamma^2.
\ee
Indeed, decay of $\bar v$, corresponding to description of behavior of $\psi$ as a modulation of $\phi$,
is the usual end goal for stability/behavior of periodic waves \cite{JNRZ3}.
However, the perturbation equations for $\bar v$ contains the shifted linear operator
$\bar {\cL}=\cA[\phi(\cdot-\gamma,\cdot)]$ in place of $\cL=\cA[\phi]$, giving, after ``centering'' to
recover the fixed linear operator $\cL$, error terms of order $\gamma$ times derivatives of $\phi$, which
are not sufficiently rapidly decaying to close a nonlinear perturbation argument.

For this reason, it is the inverse-modulated variable $v$ that is typically used in the stability analysis 
\cite{JNRZ3,DSSS,SSSU}, with bounds on $\bar v$ recovered after, by comparison with $v$, using the fact that
$v$ and $\bar v$ are related by the change of coordinates 
\be\label{choc}
x\to x-\gamma(x,t),
\ee
with $\gamma$ and all derivatives small in $L^p$, $1\leq p\leq \infty$.
This comparison is formalized in \cite[Lemma 2.7]{JNRZ3}, stating that $v$ controls $\bar v$ 
in all $L^p$ norms, provided that $\phi_x$, $\gamma$, and $\gamma_x$ are bounded in $L^\infty$
and $\gamma_x$ in $L^p$, with $\|\gamma_x\|_{L^\infty}<1$ (the latter guaranteeing invertibility of \eqref{choc}).

\subsection{Forward vs. inverse modulation bounds}\label{s:vs} 
We now make two small but useful observations.
First, we note that the argument for \cite[Lemma 2.7]{JNRZ3} gives not only $L^p$ control of $\bar v$ by $v$
but {\it equivalence of $L^p$ norms} modulo derivatives of $\gamma$, 
hence, by differentiation/induction, {\it equivalence of $H^s$ norms} modulo suitable derivatives of $\gamma$ as well.
Recall \cite[\S 2.3]{JNRZ3}, that derivatives of $\gamma$ are harmless in the derivation of nonlinear damping estimates
(see Section \ref{s:fwd} below).
We formalize this observation in the following pair of results.

\begin{lemma}[\cite{JNRZ3}]\label{switchlem}
Let $\gamma$ be bounded with $\|\gamma_x\|_{L^\infty(\RM)}<1$. Then, the change of coordinates
	$\Id-\gamma$ is invertible, with inverse $(\Id-\gamma)^{-1}=\Id + \tilde \gamma$, 
	satisfying for all $1\leq p\leq \infty$
\be\label{switcheq}
\begin{array}{rcl}
\|\psi-\phi\circ(\Id-\gamma)^{-1}\|_{L^p(\RM)}
&\le&(1+\|\gamma_x\|_{L^\infty(\RM)})^{\frac1p}\quad
\|\psi\circ(\Id-\gamma)-\phi\|_{L^p(\RM)}\\
\|\psi-\phi\circ(\Id+\gamma)\|_{L^p(\RM)}
&\le&(1+\|\gamma_x\|_{L^\infty(\RM)})^{\frac1p}\quad
\|\psi\circ(\Id-\gamma)-\phi\|_{L^p(\RM)}\\
&&+\|\phi_x\|_{L^\infty(\RM)}(1+\|\gamma_x\|_{L^\infty(\RM)})^{\frac1p}\|\gamma\|_{L^\infty(\RM)}\|\gamma_x\|_{L^p(\RM)}
\end{array}
\ee
and
\be\label{rswitcheq}
\begin{array}{rcl}
\|\psi-\phi\circ(\Id-\gamma)^{-1}\|_{L^p(\RM)}
&\ge&(1-\|\gamma_x\|_{L^\infty(\RM)})^{-\frac1p}\quad
\|\psi\circ(\Id-\gamma)-\phi\|_{L^p(\RM)}\\
\|\psi-\phi\circ(\Id+\gamma)\|_{L^p(\RM)}
&\ge&(1-\|\gamma_x\|_{L^\infty(\RM)})^{-\frac1p}\quad
\|\psi\circ(\Id-\gamma)-\phi\|_{L^p(\RM)}\\
&&-\|\phi_x\|_{L^\infty(\RM)}(1+\|\gamma_x\|_{L^\infty(\RM)})^{\frac1p}\|\gamma\|_{L^\infty(\RM)}\|\gamma_x\|_{L^p(\RM)}.
\end{array}
\ee
\end{lemma}

\begin{proof}
	We follow the argument of \cite[Lemma 2.7]{JNRZ3}.
By the implicit function theorem and boundedness of $\gamma$, the map $\Id-\gamma$ is invertible. Let us write its inverse $\Id+\tilde{\gamma}$. Since the Jacobian of $\Id+\tilde\gamma$ is bounded below by 
$(1+\|\gamma_x\|_{L^\infty(\RM)})^{-1}$, we have
	and above by $(1-\|\gamma_x\|_{L^\infty(\RM)})^{-1}$, we have
$$
\|[\psi\circ(\Id-\gamma)-\phi]\circ(\Id+\tilde\gamma)\|_{L^p(\RM)}\ \leq\ (1+\|\gamma_x\|_{L^\infty(\RM)})^{\frac1p}
\|\psi\circ(\Id-\gamma)-\phi\|_{L^p(\RM)},
$$
giving the first part of \eqref{switcheq}. The first part of \eqref{rswitcheq} follows similarly.
	Splitting $\psi-\phi\circ (\Id+\gamma)$ as
$$
\psi-\phi\circ(\Id+\gamma)\ =\ [\psi\circ(\Id-\gamma)-\phi]\circ(\Id+\tilde\gamma)\ +\ \phi\circ(\Id+\tilde\gamma)-\phi\circ(\Id+\gamma),
$$
and applying the intermediate value theorem then yields
$$
\|\phi\circ(\Id+\tilde\gamma)-\phi\circ(\Id+\gamma)\|_{L^p(\RM)}\ \leq\ \|\phi_x\|_{L^\infty(\RM)}
\|\tilde\gamma-\gamma\|_{L^p(\RM)}.
$$
But, from the identity $\tilde\gamma=\gamma_x\circ(\Id+\tilde\gamma)$ we have
$ \tilde{\gamma}(x)-\gamma(x)\ =\ \tilde\gamma(x)\ \int_0^1\ \gamma(x+t\tilde\gamma(x))\,dt $,
from which H\"older's inequality gives
$
\|\tilde{\gamma}-\gamma\|_{L^p(\RM)}^p\ \leq\ \|\tilde\gamma\|_{L^\infty(\RM)}^p\ \int_0^1\ \|\gamma_x\circ(\Id+t\tilde\gamma)\|_{L^p(\RM)}^p\,dt.
$
This gives the second part of \eqref{switcheq}, 
since $\|\tilde\gamma\|_{L^\infty(\RM)}\leq \|\gamma\|_{L^\infty(\RM)}$ and, for $t\in[0,1]$, $\Id+t\tilde\gamma$ 
is invertible with a Jacobian bounded below by $(1+\|\gamma_x\|_{L^\infty(\RM)})^{-1}$. 
The second part of \eqref{rswitcheq} follows similarly. 
\end{proof}

\br\label{asymrmk}
Note the asymmetry in bounds \eqref{switcheq}-\eqref{rswitcheq}, in that they involve derivatives of $\phi$ and $\gamma$ only, and not on $\psi$ (a potential problem, not a priori controlled) 
or $\tilde \gamma$ (harmless, controlled by $\gamma$).
\er

\bc\label{Hequiv}
Let $\gamma$ be bounded with $\|\gamma_x\|_{L^\infty(\RM)}<1$, and $\|\gamma_x\|_{H^{k+1}}\leq C_1$. 
Then, for some $C>0$.
\be\label{Hs_switcheq}
\begin{aligned}
	C^{-1}  \|\psi\circ(\Id-\gamma)-\phi\|_{H^k(\RM)}-  \|\gamma_x\|_{H^{k+1}} &\leq
	\|\psi -\phi \circ(\Id +\gamma)\|_{H^k(\RM)}\\
	&\leq C \big( \|\psi\circ(\Id-\gamma)-\phi\|_{H^k(\RM)} + \|\gamma_x\|_{H^{k+1}}\big).
\end{aligned}
\ee
\ec

\begin{proof}
	This follows readily by induction on the order of derivatives $k$, using the chain rule,
	Lemma \ref{switchlem}, and Sobolev embedding to control $L^\infty$ norms of derivatives of $\gamma$.
\end{proof}

\medskip

$\bullet$ Thus, {\it we are free to use $v$ or $\bar v$ alternatively, as is most convenient,} 
in $L^p \cap H^j$ estimates of any type of the (either inverse or forward) modulation remainder,
{\it in particular for nonlinear damping}.

\subsection{Forward damping estimate}\label{s:fwd}
Second, we note that (high-frequency) damping estimates, particularly in situations \cite{Z,JZN,HJPR}, 
do not proceed as in low-frequency estimates by separating out a centered, linearized part from a nonlinear residual,
but rather by energy estimates based on symmetric/antisymmetric structure of the equations.
Thus, there appears to be no inherent disadvantage, and in the present case considerable advantage as we shall see, in working with (uncentered) {\it forward-modulated equations} to obtain a nonlinear damping estimate.

In particular, in the case of the Lugiato-Lefever equations, the forward-modulated perturbation equations for
$\bar v$ become, writing $\psi(x,t)=\phi(x-\gamma,t)-v(x,t)$, computing derivatives
$$
\begin{aligned}
	\partial_t \psi(x,t)&=v_t(x,,t) - (\phi_t- \phi_x \partial_t \gamma)(x-\gamma,t),\\
	\partial_x \psi(x,t)&=v_t(x,,t) - (\phi_x(1+ \partial_x \gamma))(x-\gamma,t),\\
\end{aligned}
$$
etc., and substituting into \eqref{1.1}, after a computation like that of \eqref{1.3},
\be\label{f1.3}
\partial_t \bp \bar v_r\\\bar v_i\ep=
\cA[\phi(\cdot-\gamma, \cdot)]\partial_t \bp \bar v_r\\\bar v_i\ep + \bar \cN(\bar v) + 
\cR(\phi(\cdot -\gamma, \cdot), \gamma),
\ee
where $\bar\cN$ is quadratic order in $\bar v$, $\cA$ is as in \eqref{1.4}--\eqref{JL},
and $\cR$ involves products of derivatives of order $\leq 2$ of  $\phi$ agains
derivatives of order between $1$ and $2$ of $\gamma$.

Performing an energy estimate essentially identical to that in Section \ref{s:udamp}, we thus
find for the modulated energy
$$
\bar E_j (t) =
\|\partial_x^j \bar v  \|_{L^2}^2 - \frac{1}{2\beta}
\langle \cJ M[\phi(\cdot-\gamma, \cdot)]\partial_x^{j-1}\bar v, \partial_x^{j-1}\bar v \rangle,
$$
$M$ as in \eqref{M}, the estimate
$$
\partial_t \bar E_j (t) = -2\bar E_j (t) + R_1(t) + R_2(t) + R_3(t),
$$
where $R_1$, $R_2$ as before satisfy
$$
|R_1(t) |\leq  C_1\Big(\|\partial_x^{j-1}\bar v\|_{L^2} + \|\bar v\|_{L^2}\Big)(t), \qquad
|R_2(t)|\leq C_2 \|\partial_x^j\bar v\|^2_{L^2}(t)\Big( \|\partial_x^{j}\bar v\|_{L^2} + \|\bar v\|_{L^2} \Big)(t),
$$
and the new term $R_3$ satisfies 
$$
\begin{aligned}
	|R_3(t)| &\leq C_3\|\partial_{x,t}\gamma\|_{H^{j+2}}(t)
	\|\phi(\cdot-\gamma(\cdot, t), t)\|_{W^{j+3,\infty}} \|\partial_x^{j-1}\bar v\|_{L^2}(t)\\
	&\leq 
	\tilde {C}_3\Big( \|\partial_{x,t}\gamma\|_{H^{j+2}}^2 + \|\partial_x^{j-1}\bar v\|_{L^2}^2 \Big)(t) ,
\end{aligned}
$$
using integration by parts in the first line, and Young's inequality in the second.

Combining, and using Sobelev interpolation, one obtains for $\|\partial_x^j\bar v\|_{L^2}$,
$\|\partial_{x,t}\gamma\|_{H^{j+2}}$, sufficiently small
the forward-modulated nonlinear damping estimate
\be\label{fnldamp}
\partial_t \bar E(t)\leq - \theta \bar E +C(\|\bar v\|_{L^2}^2 + \|\partial_{x,t}\gamma\|_{H^{j+2}}^2)(t), 
\qquad \theta>0,
\ee
yielding after integration
\be\label{fNLgron}
\bar E(t)\leq e^{-\theta t}\bar E(0)+ C \int_0^t e^{-\theta(t-s)}(\|\bar v\|_{L^2}^2
+ \|\partial_{x,t}\gamma\|_{H^{j+2}}^2)(s) ds,
\ee
so long as $\|\partial_x^j\bar v\|_{L^2}$ and $\|\partial_{x,t}\gamma\|_{H^{j+2}}$ are
sufficiently small on $0\leq s\leq t$, hence controlling
$\|\partial_x^j\bar v\|(t)$ by exponential slaving to $\|\bar v\|_{L^2}(t)$ and $\|\partial_{x,t}\gamma\|_{H^{j+2}}$:
an exact analog of the standard inverse-modulation estimate derived in \cite[Prop. 2.5, \S2.3]{JNRZ3} 
in the parabolic case.
Summing over $o\leq j\leq k $ gives the following analog of \cite[Prop. 2.5, \S2.3]{JNRZ3}.

\bl\label{damplem}
For the forward-modulated perturbation variable $\bar v$ about periodic wave $\phi$ of (LLE),
\be\label{Hs_gron}
\|\bar v\|_{H^k}^2(t) \leq C e^{-\theta t}\|\bar v\|_{H^k}^2(0) + C \int_0^t e^{-\theta(t-s)}(\|\bar v\|_{L^2}^2
+ \|\partial_{x,t}\gamma\|_{H^{k+2}}^2)(s) ds,
\qquad \theta >0,
\ee
provided $\|\bar v\|_{H^k}(s)$ and $\|\partial_{x,t}\gamma\|_{H^{k+2}}^2(s)$ are sufficiently small
on $0\leq s\leq t$.
\el

\br\label{delicatermk}
One may check that a similar approach likewise yields a forward-modulated damping estimate in the
delicate cases treated previously in \cite{Z,JZN} by inverse-modulated damping.
\er

\section{Conclusion and applications}\label{s:conclusion}
Combining the observations of Sections \ref{s:vs} and \ref{s:fwd},
we see that a useful general strategy is to use inverse-modulated variables
to obtain linearized estimates, and forward-modulated variables for nonlinear damping estimates, the first
convenient for decay and the second for regularity.
In particular, as described in the introduction,
this allows the treatment of stability of periodic Lugiato-Lefever waves by the original techniques described
in \cite{JZ,JNRZ3} playing off linearized estimates with nonlinear damping, without the need for the 
additional techniques/estimates introduced in \cite{HJPR}.
Specifically, combining Corollary \ref{Hequiv} and Lemma \ref{damplem},
we have the following nonlinear damping estimate on the inverse-modulated perturbation variable $v$ for (LLE).

\bt\label{dampthm}
For the inverse-modulated perturbation variable $v$ about a periodic wave $\phi$ of (LLE),
\be\label{inv_Hs_gron}
\|v\|_{H^k}^2(t) \leq Ce^{-\theta t}\|v\|_{H^k}^2(0) + C \int_0^t e^{-\theta(t-s)}(\|v\|_{L^2}^2
+ \|\partial_{x,t}\gamma\|_{H^{k+2}}^2)(s) ds 
+ C\|\partial_{x}\gamma\|_{H^{k+1}}^2(t),
\ee
for $\theta >0$ provided $\|v\|_{H^k}(s)$ and $\|\partial_{x,t}\gamma\|_{H^{k+2}}^2(s)$ are sufficiently small
on $0\leq s\leq t$.
\et
\begin{proof}
From Corollary \ref{Hequiv} and Lemma \ref{damplem}, we have immediately \eqref{inv_Hs_gron}
for $\|\bar v\|_{H^k}(s)$ and $\|\partial_{x,t}\gamma\|_{H^{k+2}}^2(s)$ sufficiently small on $0\leq s\leq t$.
	Observing by a second application of Corollary \eqref{Hequiv} that $\|\bar v\|_{H^k}(s)$ is controlled by
$\| v\|_{H^k}(s)$ and $\|\partial_{x,t}\gamma\|_{H^{k+1}}(s)$, we are done.
\end{proof}

The damping estimate \eqref{inv_Hs_gron} differs from the standard one \eqref{Hs_gron} of \cite{JNRZ3} only
by the addition of the final term $C\|\partial_{x}\gamma\|_{H^{k+1}}^2(t)$, which is of the same order in
the usual nonlinear decay argument as the integral term
$C \int_0^t e^{-\theta(t-s)}\|\partial_{x,t}\gamma\|_{H^{k+2}}^2(s) ds$, hence harmless.
It follows that, indeed, the nonlinear iteration argument of \cite{HJPR} can be closed, alternatively, using
the more standard nonlinear damping estimate \eqref{inv_Hs_gron} in place of the coupled tame estimates/integration
by parts and mean value inequalities used there.

\subsection{Stability and asymptotic behavior}\label{s:appls}
The incorporation of Theorem \ref{dampthm} 
in the stability analysis of \cite{HJPR} 
for localized perturbations simplifies the arguments but does not much affect the results.
More important, having recovered the missing ingredient of nonlinear damping, we can 
apply the full machinery developed in \cite{JNRZ1,JNRZ2,JNRZ3} to obtain also new results
on stability and asymptotic behavior of periodic (LLE) waves with respect to 
{nonlocalized} perturbations.

\begin{theorem}[Stability]\label{oldmain}
Let $\phi$ be a smooth spatially periodic standing-wave solution of (LLE) that is diffusively spectrally
	stable in the sense of Schneider \cite{S1,S2}, and let $\psi$ be a perturbation such that
$
E_0:=\big\|\psi(\cdot-h_0(\cdot),0)-\phi(\cdot)\big\|_{L^1(\RM)\cap H^3(\RM)}
+\big\|\partial_x h_0\big\|_{L^1(\RM)\cap H^3(\RM)}
$
is sufficiently small, for some choice of phase modulation $h_0$
such that $h_0(-\infty)=-h_0(\infty)$.\footnote{
Achievable without loss of generality by a shift in $\bar u$ \cite{JNRZ3}.}
Then, $\psi$ exists for all $t>0$, and, for some phase function $\gamma(x,t)$ and $2\le p \le \infty$,
\ba\label{mainest}
\big\|\psi(\cdot-\gamma(\cdot,t), t)-\phi(\cdot)\big\|_{L^p(\RM)},
\quad \big\|\nabla_{x,t} \gamma(\cdot,t) \big\|_{L^{p}(\RM)}
&\lesssim E_0 (1+t)^{-\frac{1}{2}(1-1/p)},\\
\big\|\psi(\cdot-\gamma(\cdot,t), t)-\phi(\cdot)\big\|_{H^3(\RM)}
&\lesssim E_0 (1+t)^{-\frac{1}{4}},\\
\ea
and
\ba\label{andpsi}
\big\|\psi(\cdot , t)-\phi(\cdot)\big\|_{L^\infty(\RM)}, \quad
\big\| \gamma(\cdot,t) \big\|_{L^\infty(\RM)} &\lesssim E_0.
\ea
\end{theorem}

\begin{proof}
Combining the linearized estimates of \cite[\S 3]{JNRZ3} with estimate \eqref{inv_Hs_gron}
of Theorem \ref{dampthm}, and applying word for word the arguments of
\cite{JNRZ1,JNRZ2} (a special case of the more general \cite{JNRZ3}), we obtain the result;
see specifically the proof of Theorem 1.10 \cite{JNRZ3}.
\end{proof}

\br\label{linftyrmk}
We note that the more detailed linearized estimates of \cite[\S 3]{JNRZ3} 
follow from the same Bloch decomposition/spectral preparation as do those of
\cite{HJP,HJPR}; compare \cite[\S 2]{JNRZ3} and \cite[\S 2-3]{HJPR}.
A subtle difference in the two analyses is that the stronger nonlinear bounds of the
nonlinear damping approach, effectively controlling $\|v\|_{H^3}$ by $\|v\|_{L^2}$, allow the 
use of weaker linear bounds.
In particular one may obtain by Pr\"uss' Theorem exponential $H^1\to H^1$ bounds
instead of $L^2\to L^2$ bounds on the high-frequency part of the solution operator,
yielding by Sobelev embedding/interpolation exponential $H^1 \to L^p$ bounds for all $p\geq 2$.
Since $H^1$ is controlled by $L^2$ in the nonlinear iteration, this serves the same purpose
as would an $L^2\to L^\infty$ bound, allowing one to obtain $L^p$ bounds on $\bar v$ for higher norms
$p\geq 2$.
To obtain such bounds in the tame estimate framework of \cite{HJPR} would require obtaining $L^2\to L^\infty$
bounds, which may or may not be true.
\er

In Theorem \ref{oldmain},
$\psi(\cdot-\gamma(\cdot,t), t)-\phi(\cdot)$ corresponds to the modulated variable $v$ in the previous sections,
and $\psi(\cdot , t)-\phi(\cdot)$ to $\tilde v$, both decaying more slowly by factor $(1+t)^{1/2}$ than
their counterparts in the case of localized perturbations treated in \cite{HJPR}.
In particular, the phase $\gamma$ is bounded only in $L^\infty$, having infinite $L^p$ norm for any $p<\infty$.
These estimates are in fact sharp, as we now show.

Periodic standing waves occur in a one-parameter family. 
Taking the base wave $\phi$ without loss of generality 
to be period one, parametrize this family as $\phi^k(kx)$, where $k=1/X$ is the wave number, with $X$ equal to 
period.  Recall now the formal, Whitham approximation
\be\label{wref}
u(x,t)\approx
\phi^{\kappa(x,t)}(\Psi(x,t))
\ee
\cite{W,HK,Se,DSSS}
where the wave number $\kappa:=\Psi_x$ satisfies the Whitham equation 
\be\label{whitham}
\kappa_t-(\omega_0(\kappa))_x=(d(\kappa)\kappa_x)_x.
\ee
with $\omega(\kappa)\equiv 0$ the time frequency associated with
the family of periodic traveling waves- here, identically zero,
and $d$ is a diffusion term determined by formal asymptotic expansion.

Following \cite{DSSS,JNRZ2,SSSU}, define the quadratic order approximate equation
\be\label{mainwhitham}
 k_t = \ks d(\ks) k_{xx},
\ee
approximately governing a small perturbation $k=\ks \gamma_x$ of the type we seek, and define 
\be\label{h}
h(x):=\int_{-\infty}^x k(x).
\ee
Then, we have the following description of $L^p$-asymptotic behavior.

\begin{theorem}[Asymptotic behavior]\label{main}
Let $\eta>0$.
Under the assumptions of Proposition \ref{oldmain},
let $k$ satisfy
the quadratic approximant \eqref{mainwhitham} of
the second-order Whitham modulation equations \eqref{whitham} with initial data $k|_{t=0}=\ks\partial_x h_0$, 
let $h$ be as in \eqref{h}, and let $\gamma$ be the phase prescribed in the proof
	of Theorem \ref{oldmain} (see \cite{JNRZ1}). Then, for $t>0$ and $2\le p\le\infty$,
\ba\label{refinedest}
\|\psi(\cdot- \gamma(\cdot,t), t)-\phi^{\ks(1+ \gamma_x(\cdot,t))}(\,\cdot\,)\|_{L^p(\RM)}
&\lesssim E_0
\ln(2+t)\
 (1+t)^{-\frac{3}{4}},\\
\|\ks\partial_x \gamma(t)-k(t) \|_{L^p(\RM)}
&\lesssim
E_0 (1+t)^{-\frac{1}{2}(1-1/p)-\frac12+\eta}.\\
\|\gamma(t) -h(t) \|_{L^p(\RM)} &\lesssim E_0 (1+t)^{-\frac{1}{2}(1-1/p)+\eta}.\\
\ea
\end{theorem}

\begin{proof}
Again, this follows by combining the linearized estimates of \cite[\S 3]{JNRZ3} with estimate \eqref{inv_Hs_gron}
of Theorem \ref{dampthm}, and applying word for word the arguments of
\cite{JNRZ1,JNRZ2} (a special case of the more general \cite{JNRZ3}).
	See specifically the proof of Theorem 1.12 \cite{JNRZ3}.
\end{proof}

Note that $k$ and $h$ both satisfy a heat equation, with localized, and nonlocalized behavior.
When $k|_{t=0}$ has a first moment in $L^1$, its solution thus decays in all $L^p$ to a heat kernel,
while $h$  converges to an errorfunction. In particular, both $\gamma$ and $\tilde v=\psi -\phi$
have infinite $L^p$ norm for all $p<\infty$, in agreement with the estimates stated in Theorem \ref{oldmain}.
Thus, the tame estimate argument used in \cite{HJPR}, based on finiteness of $\|\gamma\|_{L^2}$ among
other things, does not suffice to treat this case.
Indeed, as described in Section \ref{s:comp}, the tame estimate approach of \cite{HJPR} gives up an
arbitrarily small amount of time-algebraic decay, so that even if carried out in $L^\infty$, where
$\gamma$ remains bounded, the estimates derived by this technique would still be unbounded and the argument not
close.

\end{document}